\documentclass[11pt,reqno]{amsart}
\usepackage{amssymb,latexsym}
\usepackage{cite}
\usepackage[height=190mm,width=130mm]{geometry}

\theoremstyle{plain}
\newtheorem{theorem}{Theorem}
\newtheorem{lemma}{Lemma}
\newtheorem{corollary}{Corollary}

\theoremstyle{definition}

\theoremstyle{remark}
\newtheorem{remark}{Remark}

\numberwithin{equation}{section} 
\input{tcilatex}

\begin{document}
\title[The generalized bi-periodic Horadam sequence]{A note on congruence
properties of the generalized bi-periodic Horadam sequence}
\author{Elif TAN}
\address{Department of Mathematics, Ankara University, Ankara, Turkey}
\email{etan@ankara.edu.tr}
\author{Ho-Hon Leung}
\address{Department of Mathematical Sciences, UAEU, Al-Ain, United Arab
Emirates}
\email{hohon.leung@uaeu.ac.ae}

\begin{abstract}
In this paper, we consider a generalization of Horadam sequence $\left\{
w_{n}\right\} $ which is defined by the recurrence $w_{n}=aw_{n-1}+cw_{n-2},$
if $n$ is even, $w_{n}=bw_{n-1}+cw_{n-2},$ if $n$ is odd with arbitrary
initial conditions $w_{0},w_{1}$ and nonzero real numbers $a,b,$ and $c.$ We
investigate some congruence properties of the generalized Horadam sequence $%
\left\{ w_{n}\right\} $.
\end{abstract}

\maketitle

\section{Introduction}

The generalized bi-periodic Horadam sequence $\left\{ w_{n}\right\} $ is
defined by the recurrence relation%
\begin{equation*}
w_{n}=\left\{ 
\begin{array}{ll}
aw_{n-1}+cw_{n-2}, & \mbox{ if }n\mbox{ is even} \\ 
bw_{n-1}+cw_{n-2}, & \mbox{ if }n\mbox{ is odd}%
\end{array}%
\right. ,\text{ }n\geq 2
\end{equation*}%
with arbitrary initial conditions $w_{0},w_{1}$ and nonzero real numbers $%
a,b $ and $c$. It is emerged as a generalization of the best known sequences
in the literature, such as the Horadam sequence, the Fibonacci\&Lucas
sequence, the $k$-Fibonacci\&$k$-Lucas sequence, the Pell\&Pell-Lucas
sequence, the Jacobsthal\& Jacobsthal-Lucas sequence, etc. Similar to the
notation of the classical Horadam sequence \cite{Horadam}, we write $\left\{
w_{n}\right\} := $ $\{w_{n}\left( w_{0},w_{1};a,b,c\right) \}.$ In
particular, using this notation, we define $\{u_{n}\}=\{w_{n}\left(
0,1;a,b,c\right) \}$ and $\{v_{n}\}=\{w_{n}\left( 2,b;a,b,c\right) \}$ as
the \textit{generalized bi-periodic Fibonacci sequence} and the \textit{%
generalized bi-periodic Lucas sequence, }respectively\textit{.} For the
basic properties of the generalized bi-periodic Horadam sequence and some
special cases of this sequence, see \cite{Edson, Yayenie, Sahin, Panario,
Bilgici, Tan1, Tan3, Tan5, Tan-Leung}.

On the other hand, it is important to investigate the congruence properties
of different integer sequences. Several methods can be applied to produce
identities for the Fibonacci and Lucas sequences. For example, Carlitz and
Ferns \cite{Carlitz} used polynomial identities in conjunction with the
Binet formula to generate new identities for these sequences. The method of
Carlitz and Ferns was used by several authors to obtain analogous results
for the generalized Fibonacci and Lucas sequences, see \cite{Zhang0, Kilic}.
On the other hand, Keskin and Siar \cite{Siar} obtained some number
theoretic properties of the generalized Fibonacci and Lucas numbers by using
matrix method. Morover, Hsu and Maosen \cite{Hsu} and Zhang \cite{Zhang1}
applied an operator method to establish some of these properties. Recently,
Yang and Zhang \cite{Zhang} have studied some congruence relations for the
bi-periodic Fibonacci and Lucas sequences by using operator method. But some
of the results that are obtained by the operator method are incorrect. In
this study, by using the method of Carlitz and Ferns \cite{Carlitz}, we give
more general identities involving the generalized bi-periodic Horadam
sequences and derive some congruence properties of the generalized
bi-periodic Horadam numbers. In particular, our results include the
corrected version of some of the results in \cite{Zhang}.

The outline of this paper as follows: In Section $2$, we give some basic
properties of the generalized bi-periodic Horadam sequence. In Section $3$,
we give some binomial identities and congruence relations for the
generalized bi-periodic Horadam sequence by using the method of Carlitz and
Ferns \cite{Carlitz}.

\section{Some preliminary results for the sequence $\{w_{n}\}$}

In this section, we give some basic properties of the bi-periodic Horadam
sequences.

The Binet formula of the sequence $\{u_{n}\}$ is%
\begin{equation}
u_{n}=\frac{a^{\xi \left( n+1\right) }}{\left( ab\right) ^{\left\lfloor 
\frac{n}{2}\right\rfloor }}\left( \frac{\alpha ^{n}-\beta ^{n}}{\alpha
-\beta }\right)  \label{3}
\end{equation}%
which can be obtained by \cite[Theorem 8]{Yayenie}. Here $\alpha $ and $%
\beta $ are the roots of the polynomial $x^{2}-abx-abc,$ that is, $\alpha =%
\frac{ab+\sqrt{a^{2}b^{2}+4abc}}{2}$ and $\beta =\frac{ab-\sqrt{%
a^{2}b^{2}+4abc}}{2}$, and $\xi \left( n\right) =n-2\left\lfloor \frac{n}{2}%
\right\rfloor $ is the parity function, i.e., $\xi \left( n\right) =0$ when $%
n$ is even and $\xi \left( n\right) =1$ when $n$ is odd. Let assume $\Delta
:=a^{2}b^{2}+4abc\neq 0.$ Also we have $\alpha +\beta =ab,$ $\alpha -\beta =%
\sqrt{a^{2}b^{2}+4abc}$ and $\alpha \beta =-abc.$

\begin{lemma}
\label{l1}For any integer $n>0,$ we have%
\begin{equation*}
w_{n}=u_{n}w_{1}+c\left( \frac{b}{a}\right) ^{\xi \left( n\right)
}u_{n-1}w_{0}.
\end{equation*}
\end{lemma}

By\ using Lemma \ref{l1} and the Binet formula of $\left\{ u_{n}\right\} $
in (\ref{3}), we can easily obtain the Binet formula of the sequence $%
\{w_{n}\}.$ We note that the extended Binet formula for the general case of
this sequence was given in \cite[Theorem 9]{Panario}. But here we express
the Binet formula of the sequence $\{w_{n}\}$ in a different manner, that
is, our $\alpha $ and $\beta $ are different from the roots which are used
in \cite{Panario}.

\begin{theorem}
\label{t1}(Binet Formula) For $n>0,$ we have%
\begin{equation*}
w_{n}=\frac{a^{\xi \left( n+1\right) }}{\left( ab\right) ^{\left\lfloor 
\frac{n}{2}\right\rfloor }}\left( A\alpha ^{n}-B\beta ^{n}\right) ,
\end{equation*}%
where $A:=\frac{w_{1}-\frac{\beta }{a}w_{0}}{\alpha -\beta }$and $B:=\frac{%
w_{1}-\frac{\alpha }{a}w_{0}}{\alpha -\beta }.$
\end{theorem}

\begin{proof}
By\ using Lemma \ref{l1} and (\ref{3}), we get the desired result.
\end{proof}

By taking initial conditions $w_{0}=2,w_{1}=b$ in Theorem \ref{t1}, we
obtain the Binet formula for the sequence $\{v_{n}\}$ as follows:%
\begin{equation}
v_{n}=\frac{a^{-\xi \left( n\right) }}{\left( ab\right) ^{\left\lfloor \frac{%
n}{2}\right\rfloor }}\left( \alpha ^{n}+\beta ^{n}\right) .  \label{4}
\end{equation}%
Also by\ using Lemma \ref{l1}, we have $v_{n}=bu_{n}+2c\left( \frac{b}{a}%
\right) ^{\xi \left( n\right) }u_{n-1}.$ Thus we get a relation between the
generalized bi-periodic Fibonacci and the generalized bi-periodic Lucas
numbers as:%
\begin{equation}
v_{n}=\left( \frac{b}{a}\right) ^{\xi \left( n\right) }\left(
u_{n+1}+cu_{n-1}\right) .  \label{5}
\end{equation}%
It should be noted that the generalized Lucas sequence $\left\{
t_{n}\right\} $ in \cite{Zhang} is a special case of the generalized
bi-periodic Horadam sequence. That is, $\{t_{n}\}=\{w_{n}\left(
2a,ab;a,b,1\right) \}.$

The generating function of the sequence $\left\{ w_{n}\right\} $ is%
\begin{equation}
G\left( x\right) =\frac{w_{0}+w_{1}x+\left( aw_{1}-\left( ab+c\right)
w_{0}\right) x^{2}+c\left( bw_{0}-w_{1}\right) x^{3}}{1-\left( ab+2c\right)
x^{2}+c^{2}x^{4}},  \label{6}
\end{equation}%
which can be obtained from \cite[Theorem 6]{Panario}.

Also we need the following identity which can be found in \cite{Tan-Leung}:%
\begin{equation}
u_{mn+r}=\frac{a^{1-\xi \left( mn+r\right) }}{\left( ab\right)
^{\left\lfloor \frac{mn+r}{2}\right\rfloor }}\sum_{i=0}^{n}\dbinom{n}{i}%
c^{n-i}u_{m}^{i}u_{m-1}^{n-i}u_{i+r}\delta \lbrack m,n,r,i]  \label{7}
\end{equation}%
where $\delta \lbrack m,n,r,i]:=\left( ab\right) ^{\left\lfloor \frac{i+r}{2}%
\right\rfloor +n\left\lfloor \frac{m}{2}\right\rfloor }a^{-\xi \left(
m+1\right) i-1+\xi \left( i+r\right) }b^{\xi \left( m\right) \left(
n-i\right) }.$

\section{Main results}

To extend the results in \cite[Theorem 4.7, Theorem 4.9, Theorem 4.11,
Theorem 4.13]{Zhang}, we give the following theorem. Also we assume that $%
a,b $ and $c$ are positive integers.

\begin{theorem}
\label{t2}For any nonnegative integers $n,r$ and $m$ with $m>1$, we have%
\begin{equation*}
w_{mn+r}=\frac{a^{1-\xi \left( mn+r\right) }}{\left( ab\right)
^{\left\lfloor \frac{mn+r}{2}\right\rfloor }}\sum_{i=0}^{n}\dbinom{n}{i}%
c^{n-i}u_{m}^{i}u_{m-1}^{n-i}w_{i+r}\delta \lbrack m,n,r,i]
\end{equation*}%
where $\delta \lbrack m,n,r,i]:=\left( ab\right) ^{\left\lfloor \frac{i+r}{2}%
\right\rfloor +n\left\lfloor \frac{m}{2}\right\rfloor }a^{-\xi \left(
m+1\right) i-1+\xi \left( i+r\right) }b^{\xi \left( m\right) \left(
n-i\right) }.$
\end{theorem}

\begin{proof}
Similar to the relation $\gamma ^{n}=\gamma F_{n}+F_{n-1}$ for the classical
Fibonacci numbers, where $\gamma $ is one of the root of the equation $%
x^{2}-x-1=0,$ we have 
\begin{equation*}
\alpha ^{m}=a^{-1}a^{\frac{m+\xi (m)}{2}}b^{\frac{m-\xi (m)}{2}}\alpha
u_{m}+ca^{\frac{m-\xi (m)}{2}}b^{\frac{m+\xi (m)}{2}}u_{m-1}
\end{equation*}%
and%
\begin{equation*}
\beta ^{m}=a^{-1}a^{\frac{m+\xi (m)}{2}}b^{\frac{m-\xi (m)}{2}}\beta
u_{m}+ca^{\frac{m-\xi (m)}{2}}b^{\frac{m+\xi (m)}{2}}u_{m-1}.
\end{equation*}%
By using the binomial theorem, we have%
\begin{equation*}
\alpha ^{mn}=\sum_{i=0}^{n}\dbinom{n}{i}a^{-i}a^{i\frac{m+\xi (m)}{2}+\left(
n-i\right) \frac{m-\xi (m)}{2}}b^{i\frac{m-\xi (m)}{2}+\left( n-i\right) 
\frac{m+\xi (m)}{2}}c^{n-i}u_{m}^{i}u_{m-1}^{n-i}\alpha ^{i},
\end{equation*}%
\begin{equation*}
\beta ^{mn}=\sum_{i=0}^{n}\dbinom{n}{i}a^{-i}a^{i\frac{m+\xi (m)}{2}+\left(
n-i\right) \frac{m-\xi (m)}{2}}b^{i\frac{m-\xi (m)}{2}+\left( n-i\right) 
\frac{m+\xi (m)}{2}}c^{n-i}u_{m}^{i}u_{m-1}^{n-i}\beta ^{i}.
\end{equation*}

Multiplying both sides of the above equalities by $A\alpha ^{r}$ and $B\beta
^{r},$ respectively, and using the Binet formula of $\left\{ w_{n}\right\} ,$
we get%
\begin{eqnarray*}
&&\left( A\alpha ^{mn+r}-B\beta ^{mn+r}\right) \\
&=&a^{-\xi \left( mn+r+1\right) }\left( ab\right) ^{\left\lfloor \frac{mn+r}{%
2}\right\rfloor }w_{mn+r} \\
&=&\sum_{i=0}^{n}\dbinom{n}{i}\left( ab\right) ^{\left\lfloor \frac{i+r}{2}%
\right\rfloor +n\left\lfloor \frac{m}{2}\right\rfloor }a^{-\xi \left(
m+1\right) i-1+\xi \left( i+r\right) }b^{\xi \left( m\right) \left(
n-i\right) }c^{n-i}u_{m}^{i}u_{m-1}^{n-i}w_{i+r}.
\end{eqnarray*}

which gives the desired result. It can be expressed as%
\begin{equation*}
w_{mn+r}=\sum_{i=0}^{n}\dbinom{n}{i}\left( \frac{b}{a}\right) ^{\frac{\xi
\left( mn+r\right) -2i\xi (m)+i-\xi \left( i+r\right) +n\xi (m)}{2}%
}c^{n-i}u_{m}^{i}u_{m-1}^{n-i}w_{i+r}.
\end{equation*}

We note that it can also be obtained by using Lemma \ref{l1} and the
identity (\ref{7}) as:%
\begin{eqnarray*}
&&w_{mn+r} \\
&=&w_{1}\frac{a^{1-\xi \left( mn+r\right) }}{\left( ab\right) ^{\left\lfloor 
\frac{mn+r}{2}\right\rfloor }}\sum_{i=0}^{n}\dbinom{n}{i}%
c^{n-i}u_{m}^{i}u_{m-1}^{n-i}u_{i+r}\delta \lbrack m,n,r,i] \\
&&+w_{0}c\left( \frac{b}{a}\right) ^{\xi \left( mn+r\right) }\frac{a^{1-\xi
\left( mn+r-1\right) }}{\left( ab\right) ^{\left\lfloor \frac{mn+r-1}{2}%
\right\rfloor }}\sum_{i=0}^{n}\dbinom{n}{i}%
c^{n-i}u_{m}^{i}u_{m-1}^{n-i}u_{i+r-1}\delta \lbrack m,n,r-1,i]\text{ \ \ \
\ }
\end{eqnarray*}%
\begin{eqnarray*}
&=&\frac{a^{1-\xi \left( mn+r\right) }}{\left( ab\right) ^{\left\lfloor 
\frac{mn+r}{2}\right\rfloor }}\sum_{i=0}^{n}\dbinom{n}{i}%
c^{n-i}u_{m}^{i}u_{m-1}^{n-i}\delta \lbrack m,n,r,i]\left(
w_{1}u_{i+r}+cw_{0}\left( \frac{b}{a}\right) ^{\xi \left( i+r\right)
}u_{i+r-1}\right) \\
&=&\frac{a^{1-\xi \left( mn+r\right) }}{\left( ab\right) ^{\left\lfloor 
\frac{mn+r}{2}\right\rfloor }}\sum_{i=0}^{n}\dbinom{n}{i}%
c^{n-i}u_{m}^{i}u_{m-1}^{n-i}w_{i+r}\delta \lbrack m,n,r,i].
\end{eqnarray*}
\end{proof}

By considering the identity $\xi (mn+r)=\xi (mn)+\xi (r)-2\xi (mn)\xi (r),$
we have the following corollary.

\begin{remark}
When $m$ and $r$ are all even and $c=1$ in Theorem \ref{t2}, we obtain the
identity%
\begin{equation*}
w_{2mn+2r}=\dsum\limits_{i=0}^{n}\binom{n}{i}\left( \frac{b}{a}\right) ^{%
\frac{i-\xi \left( i\right) }{2}}u_{2m}^{i}u_{2m-1}^{n-i}w_{i+2r}.
\end{equation*}%
Thus, the result in \cite[Theorem 4.7]{Zhang} can be corrected by
multiplying the right side of the equation by $\left( \frac{b}{a}\right) ^{%
\frac{i-\xi \left( i\right) }{2}}.$ The other results in \cite[Theorem 4.9,
Theorem 4.11, Theorem 4.13]{Zhang} can be corrected similarly.
\end{remark}

\begin{corollary}
\label{c1}For $m,n,r>0,$ we have%
\begin{equation*}
w_{mn+r}-\left( \frac{b}{a}\right) ^{\xi (m)\left( \frac{n+\xi (n)}{2}%
\right) -\xi \left( mn\right) \xi \left( r\right) }c^{n}u_{m-1}^{n}w_{r}%
\text{ }\equiv 0\left( \func{mod}u_{m}\right) .
\end{equation*}
\end{corollary}

Now we give a generalization of the Ruggles identity \cite{Ruggles} which
also generalizes the identities in \cite[Theorem 2.2 (3-4)]{Zhang} and \cite[%
Theorem 1]{Yayenie}. Then we give a related binomial identity for the
generalized bi-periodic Horadam sequence\textit{.}

For $n\geq 0$ and $k\geq 1,$ the Ruggles identity \cite{Ruggles} is given by%
\begin{equation*}
F_{n+2k}=L_{k}F_{n+k}+\left( -1\right) ^{k+1}F_{n},
\end{equation*}%
where $\{F_{n}\}$ and $\{L_{n}\}$ are the Fibonacci and Lucas numbers,
respectively. Horadam \cite{Horadam} generalized this result to a general
second order recurrence relation%
\begin{equation*}
W_{n+2k}=V_{k}W_{n+k}+\left( -1\right) ^{k+1}q^{k}W_{n},
\end{equation*}%
where $W_{k}=pW_{k-1}+qW_{k-2}$ with arbitrary initial conditions $W_{0}$
and $W_{1}.$ The sequence $\left\{ V_{k}\right\} $ satisfies the same
recurrence relation as the sequence $\left\{ W_{k}\right\} ,$ but it begins
with $V_{0}=2,V_{1}=p.$

A generalization of Ruggles identity can be given in the following lemma.

\begin{lemma}
\label{l2}For integers $n\geq 0$ and $k\geq 1,$ we have%
\begin{equation*}
w_{n+2k}=\left( \frac{a}{b}\right) ^{\xi (n+1)\xi (k)}v_{k}w_{n+k}-\left(
-c\right) ^{k}w_{n}
\end{equation*}%
where $\left\{ w_{n}\right\} $ is the \textit{generalized bi-periodic
Horadam sequence} and $\{v_{n}\}$ is the \textit{generalized bi-periodic
Lucas sequence}.
\end{lemma}

\begin{proof}
It can be obtained simply by the Binet formula of $\left\{ w_{n}\right\} .$
\end{proof}

\begin{theorem}
\label{t3}For nonnegative integers $n,r$ and $m$ with $m>1$, we have%
\begin{equation*}
w_{2mn+r}=\dsum\limits_{i=0}^{n}\binom{n}{i}\left( -1\right) ^{\left(
m+1\right) \left( n-i\right) }\left( \frac{a}{b}\right) ^{\xi (m)\left( 
\frac{i+\xi \left( i\right) }{2}\right) -\xi (im)\xi (r)}c^{m\left(
n-i\right) }v_{m}^{i}w_{im+r}.
\end{equation*}
\end{theorem}

\begin{proof}
From the Binet formula of $\left\{ v_{k}\right\} $ and $\alpha \beta =-abc,$
it is clear to see that 
\begin{equation*}
\alpha ^{2m}=a^{\frac{m+\xi (m)}{2}}b^{\frac{m-\xi (m)}{2}}v_{m}\alpha
^{m}-\left( -abc\right) ^{m}.
\end{equation*}%
By using the binomial theorem, we have%
\begin{equation*}
\alpha ^{2mn}=\sum_{i=0}^{n}\dbinom{n}{i}a^{i\frac{m+\xi (m)}{2}}b^{i\frac{%
m-\xi (m)}{2}}\left( -1\right) ^{\left( m+1\right) \left( n-i\right) }\left(
abc\right) ^{m\left( n-i\right) }v_{m}^{i}\alpha ^{im}.
\end{equation*}%
Similarly, we have%
\begin{equation*}
\beta ^{2mn}=\sum_{i=0}^{n}\dbinom{n}{i}a^{i\frac{m+\xi (m)}{2}}b^{i\frac{%
m-\xi (m)}{2}}\left( -1\right) ^{\left( m+1\right) \left( n-i\right) }\left(
abc\right) ^{m\left( n-i\right) }v_{m}^{i}\beta ^{im}.
\end{equation*}

Multiplying both sides of the above equalities by $A\alpha ^{r}$ and $B\beta
^{r},$ respectively, and using the Binet formula of $\left\{ w_{n}\right\} ,$
we get%
\begin{eqnarray*}
&&\left( A\alpha ^{2mn+r}-B\beta ^{2mn+r}\right) \\
&=&a^{-\xi \left( 2mn+r+1\right) }\left( ab\right) ^{\left\lfloor \frac{2mn+r%
}{2}\right\rfloor }w_{2mn+r} \\
&=&\sum_{i=0}^{n}\dbinom{n}{i}a^{i\frac{m+\xi (m)}{2}}b^{i\frac{m-\xi (m)}{2}%
}\left( -1\right) ^{\left( m+1\right) \left( n-i\right) }\left( abc\right)
^{m\left( n-i\right) }v_{m}^{i}\left( A\alpha ^{im+r}-B\beta ^{im+r}\right) .
\end{eqnarray*}%
Thus, again by using the identity $\xi (mn+r)=\xi (mn)+\xi (r)-2\xi (mn)\xi
(r),$ we have%
\begin{eqnarray*}
w_{2mn+r} &=&\sum_{i=0}^{n}\dbinom{n}{i}\left( -1\right) ^{\left( m+1\right)
\left( n-i\right) }c^{m\left( n-i\right) }v_{m}^{i}w_{im+r} \\
&&\times a^{i\frac{m+\xi (m)}{2}}b^{i\frac{m-\xi (m)}{2}}\left( ab\right)
^{m\left( n-i\right) -\left\lfloor \frac{2mn+r}{2}\right\rfloor
+\left\lfloor \frac{im+r}{2}\right\rfloor }a^{\xi \left( 2mn+r+1\right)
}a^{-\xi \left( im+r+1\right) }
\end{eqnarray*}%
which gives the desired result.
\end{proof}

\begin{remark}
Since $\left( \frac{a}{b}\right) ^{\xi (r+1)\xi (m)i+\left( -1\right)
^{r+1}\xi (m)\left( \frac{i-\xi \left( i\right) }{2}\right) }=\left( \frac{a%
}{b}\right) ^{\xi (m)\left( \frac{i+\xi \left( i\right) }{2}\right) -\xi
(im)\xi (r)},$ the results in \cite[Theorem 4.3, Theorem 4.5]{Zhang} can be
corrected by multiplying the right side of the equations by $\left( \frac{a}{%
b}\right) ^{\left( -1\right) ^{r+1}\left( \frac{i-\xi \left( i\right) }{2}%
\right) }.$
\end{remark}

\begin{corollary}
\label{c2}For $m,n,r>0,$ we have%
\begin{equation*}
w_{2mn+r}-\left( -1\right) ^{\left( m+1\right) n}c^{mn}w_{r}\equiv 0\left( 
\func{mod}v_{m}\right) .
\end{equation*}
\end{corollary}

Note that for $m=2$ and $m=3$, Corollary \ref{c2} gives the results in \cite[%
Corollary 4.2, Corollary 4.4, Corollary 4.6]{Zhang}. Also for the case of
generalized Fibonacci and Lucas sequences, it gives the results in \cite[%
3.3. Corollary]{Siar}.

\begin{lemma}
\label{l3}For $m,r>0$, we have%
\begin{eqnarray*}
&&-\left( -abc\right) ^{m+r}+a^{\frac{r+\xi (r)}{2}}b^{\frac{r-\xi (r)}{2}%
}v_{r}\left( -abc\right) ^{m}z^{r}+z^{2\left( m+r\right) } \\
&=&z^{m+2r}a^{\frac{m+\xi (m)}{2}}b^{\frac{m-\xi (m)}{2}}v_{m}
\end{eqnarray*}%
where $z$ is either $\alpha $ or $\beta .$
\end{lemma}

\begin{proof}
From the Binet formula of $\left\{ v_{n}\right\} $ and $\alpha \beta =-abc,$
it is clear to see that $z^{2r}=a^{\frac{r+\xi (r)}{2}}b^{\frac{r-\xi (r)}{2}%
}v_{r}z^{r}-\left( -abc\right) ^{r}.$ Thus we have%
\begin{eqnarray*}
&&-\left( -abc\right) ^{m+r}+a^{\frac{r+\xi (r)}{2}}b^{\frac{r-\xi (r)}{2}%
}v_{r}\left( -abc\right) ^{m}z^{r}+z^{2\left( m+r\right) } \\
&=&\left( -abc\right) ^{m}\left( a^{\frac{r+\xi (r)}{2}}b^{\frac{r-\xi (r)}{2%
}}v_{r}z^{r}-\left( -abc\right) ^{r}\right) +z^{2\left( m+r\right) } \\
&=&\left( -abc\right) ^{m}z^{2r}+z^{2\left( m+r\right) } \\
&=&z^{m+2r}\left( \left( -abc\right) ^{m}z^{-m}+z^{m}\right) \\
&=&z^{m+2r}\left( \beta ^{m}+\alpha ^{m}\right) \\
&=&z^{m+2r}a^{\frac{m+\xi (m)}{2}}b^{\frac{m-\xi (m)}{2}}v_{m}.
\end{eqnarray*}
\end{proof}

\begin{theorem}
\label{t4}For $n,m,r>0$, we have%
\begin{eqnarray*}
&&-\left( -c\right) ^{m+r}w_{n}+\left( -c\right) ^{m}\left( \frac{a}{b}%
\right) ^{\xi (r)\xi (n+1)}v_{r}w_{r+n}+w_{2\left( m+r\right) +n} \\
&=&\left( \frac{a}{b}\right) ^{\xi (m)\xi (n+1)}v_{m}w_{m+2r+n}.
\end{eqnarray*}
\end{theorem}

\begin{proof}
From Lemma \ref{l3}, we have%
\begin{eqnarray}
&&-\left( -abc\right) ^{m+r}+a^{\frac{r+\xi (r)}{2}}b^{\frac{r-\xi (r)}{2}%
}v_{r}\left( -abc\right) ^{m}\alpha ^{r}+\alpha ^{2\left( m+r\right) } 
\notag \\
&=&\alpha ^{m+2r}a^{\frac{m+\xi (m)}{2}}b^{\frac{m-\xi (m)}{2}}v_{m}.
\label{*}
\end{eqnarray}%
Similarly, we have%
\begin{eqnarray}
&&-\left( -abc\right) ^{m+r}+a^{\frac{r+\xi (r)}{2}}b^{\frac{r-\xi (r)}{2}%
}v_{r}\left( -abc\right) ^{m}\beta ^{r}+\beta ^{2\left( m+r\right) }  \notag
\\
&=&\beta ^{m+2r}a^{\frac{m+\xi (m)}{2}}b^{\frac{m-\xi (m)}{2}}v_{m}.
\label{**}
\end{eqnarray}%
By multiplying both sides of the equations (\ref{*}) and (\ref{**}) by $%
A\alpha ^{n}$ and $B\beta ^{n},$ respectively, we get%
\begin{eqnarray*}
&&-\left( A\alpha ^{n}-B\beta ^{n}\right) \left( -abc\right) ^{m+r} \\
&&+a^{\frac{r+\xi (r)}{2}}b^{\frac{r-\xi (r)}{2}}v_{r}\left( -abc\right)
^{m}\left( A\alpha ^{r+n}-B\beta ^{r+n}\right) \\
&&+\left( A\alpha ^{2\left( m+r\right) +n}-B\beta ^{2\left( m+r\right)
+n}\right) \\
&=&a^{\frac{m+\xi (m)}{2}}b^{\frac{m-\xi (m)}{2}}v_{m}\left( A\alpha
^{m+2r+n}-B\beta ^{m+2r+n}\right) .
\end{eqnarray*}%
Then by using the Binet formula of $\left\{ w_{n}\right\} $, we have 
\begin{eqnarray*}
&&-\left( ab\right) ^{m+r}\left( -c\right) ^{m+r}a^{\frac{n+\xi (n)}{2}}b^{%
\frac{n-\xi (n)}{2}}w_{n} \\
&&+\left( -c\right) ^{m}\left( ab\right) ^{m+r}a^{\frac{n+\xi (r+n)+\xi (r)}{%
2}}b^{\frac{n-\xi (r+n)-\xi (r)}{2}}v_{r}w_{r+n} \\
&&+\left( ab\right) ^{m+r}a^{\frac{n+\xi (n)}{2}}b^{\frac{n-\xi (n)}{2}%
}w_{2\left( m+r\right) +n} \\
&=&\left( ab\right) ^{m+r}a^{\frac{n+\xi (m+n)+\xi (m)}{2}}b^{\frac{n-\xi
(m+n)-\xi (m)}{2}}v_{m}w_{m+2r+n}.
\end{eqnarray*}%
By considering the identity $\xi (mn+r)=\xi (mn)+\xi (r)-2\xi (mn)\xi (r),$
we get the desired result.
\end{proof}

If we take $r=1,m=2,c=1$ in Theorem \ref{t4}, we get%
\begin{equation*}
\left( ab+2\right) w_{n+4}=w_{n}+a^{\xi (n+1)}b^{\xi (n)}w_{n+1}+w_{n+6}
\end{equation*}%
which reduces to the identity%
\begin{equation*}
\left( ab+1\right) w_{n+4}=w_{n}+a^{\xi (n+1)}b^{\xi (n)}w_{n+1}+a^{\xi
(n+1)}b^{\xi (n)}w_{n+5}
\end{equation*}%
in \cite[Theorem 4.15, Theorem 4.17]{Zhang}.

\begin{theorem}
\label{t5}The symbol $\dbinom{n}{i,j}$is defined by $\dbinom{n}{i,j}:=\frac{%
n!}{i!j!\left( n-i-j\right) !}.$ For $n,m,r,d>0,$ we have%
\begin{eqnarray}
w_{\left( m+2r\right) n+d} &=&v_{m}^{-n}\sum_{i+j+s=n}\dbinom{n}{i,j}\left(
-1\right) ^{s}\left( -c\right) ^{mj+\left( m+r\right) s}v_{r}^{j}w_{2\left(
m+r\right) i+rj+d}  \notag \\
&&\times \left( \frac{a}{b}\right) ^{\xi (r)\frac{j+\xi \left( j\right) }{2}%
-\xi (m)\frac{n+\xi \left( n\right) }{2}-\xi \left( rj\right) \xi \left(
d\right) +\xi \left( mn\right) \xi \left( d\right) }  \label{s2}
\end{eqnarray}%
and%
\begin{eqnarray}
w_{2\left( m+r\right) n+d} &=&\sum_{i+j+s=n}\dbinom{n}{i,j}\left( -1\right)
^{j}\left( -c\right) ^{s\left( m+r\right) +mj}v_{m}^{i}v_{r}^{j}w_{\left(
m+2r\right) i+rj+d}  \notag \\
&&\times \left( \frac{a}{b}\right) ^{\xi (m)\frac{i+\xi \left( i\right) }{2}%
+\xi (r)\frac{j+\xi \left( j\right) }{2}-\xi \left( mi\right) \xi \left(
rj\right) -\xi \left( mi+rj\right) \xi \left( d\right) }.  \label{s3}
\end{eqnarray}
\end{theorem}

\begin{proof}
By using Lemma \ref{l3} and the multinomial theorem, we obtain the following
identities:%
\begin{eqnarray*}
&&a^{n\frac{m+\xi (m)}{2}}b^{n\frac{m-\xi (m)}{2}}v_{m}^{n}z^{\left(
m+2r\right) n} \\
&=&\sum_{i+j+s=n}\dbinom{n}{i,j}\left( -1\right) ^{s}\left( -abc\right)
^{s\left( m+r\right) +mj}a^{j\frac{r+\xi (r)}{2}}b^{j\frac{r-\xi (r)}{2}%
}v_{r}^{j}z^{2\left( m+r\right) i+rj}\text{\ \ \ \ \ \ \ \ \ \ \ \ \ \ \ \ \
\ \ \ \ \ \ }
\end{eqnarray*}%
and%
\begin{eqnarray*}
&&z^{2\left( m+r\right) n} \\
&=&\sum_{i+j+s=n}\dbinom{n}{i,j}\left( -1\right) ^{j}\left( -abc\right)
^{s\left( m+r\right) +mj}a^{\frac{im+i\xi (m)+jr+j\xi (r)}{2}}b^{\frac{%
im-i\xi (m)+jr-j\xi (r)}{2}}v_{m}^{i}v_{r}^{j}z^{\left( m+2r\right) i+rj}.
\end{eqnarray*}%
By multiplying both sides in the preceding equalities by $z^{d}$ and using
the Binet formula of $\left\{ w_{n}\right\} ,$ we have (\ref{s2}) and (\ref%
{s3}), respectively.
\end{proof}

We note that for the computational simplicity, the equation (\ref{s3}) is
more practical than the equation (\ref{s2}).

From (\ref{s3}), by using the decomposition%
\begin{equation*}
\sum_{i+j+s=n}=\sum_{i+j+s=n,i=0}+\sum_{i+j+s=n,i\neq 0}
\end{equation*}%
and Theorem \ref{t3}, we get the following corollary.

\begin{corollary}
\label{c3}For $n,m,r,d>0$, we have%
\begin{equation*}
w_{2\left( m+r\right) n+d}-\left( -1\right) ^{n\left( m+1\right)
}c^{mn}w_{2rn+d}\equiv 0\left( \func{mod}v_{m}\right) .
\end{equation*}
\end{corollary}

\section{Acknowledgement}

The first author is grateful to Dr. Mohamed Salim for the arrangement of her
visit to United Arab Emirates University (UAEU) in February 2019. It is
supported by UAEU UPAR Grant G00002599 (Fund No. 31S314).

\end{document}